\documentclass[10pt]{amsproc}
\usepackage{amssymb}
\usepackage{graphicx,color}
\usepackage{amsmath}
\usepackage{amscd} %% Package for commutative diagrams
\theoremstyle{plain}
 \newtheorem{thm}{Theorem}[section]
 \newtheorem{prop}{Proposition}[section]
 \newtheorem{lem}{Lemma}[section]
 
 \newtheorem{claim}{Claim}[section]
\theoremstyle{definition}
 
 \newtheorem{dfn}{Definition}[section]
\theoremstyle{remark}
 \newtheorem{rem}{Remark}[section] 
 \numberwithin{equation}{section}

\renewcommand{\setminus}{\smallsetminus}

\newcommand{\pos}{\mathrm{pos}}
\newcommand{\supp}{\mathrm{supp} \, }
\title[On the interior of projections of planar self-similar sets]
{On the interior of projections of planar self-similar sets}

\author[Y.\ Takahashi]{YUKI TAKAHASHI}

\address{Department of Mathematics, Bar-Ilan University, Ramat-Gan, 5290002, Israel}

\email{takahashi@math.biu.ac.il}

\thanks{Y.\ T. \ was supported by the Israel Science Foundation grant 396/15 (PI: B.\ Solomyak).}

\date{today}

\begin{document}

%{\begin{flushleft}\baselineskip9pt\scriptsize
%PUBLICATIONS DE L'INSTITUT MATH\'EMATIQUE\newline
%Nouvelle s\'erie, tome 87(101) (2010), od--do \hfill DOI:
%\end{flushleft}}
\vspace{18mm}
\setcounter{page}{1}
\thispagestyle{empty}

\begin{abstract}
We consider projections of planar self-similar sets, and show that one can create nonempty interior  
in the projections by applying arbitrary small perturbations, 
if the self-similar set satisfies the open set condition and has Hausdorff dimension greater than $1$. 
\end{abstract}

\maketitle

\section{Introduction and main results}
\subsection{Planar self-similar sets and their projections}
The study of orthogonal projections of fractal sets has a long history, 
dating back to Marstrand's projection theorem \cite{Marstrand}. 
Ever since, it has been considered in many papers and in many different settings 
(e.g. \cite{Ber}, \cite{Eroglu}, \cite{Falconer2}, \cite{Farkas}, \cite{Ferguson}, \cite{Hochman2}, \cite{Hu}, \cite{Kenyon}, \cite{SS}). 
For a broad view of the subject, the reader is referred to the survey \cite{Falconer}. 
For the recent developments of the study, see \cite{Shmerkin}. 

Recall that, an \emph{Iterated Function System} (IFS) on $\mathbb{R}^n$ is a finite collection 
$\mathcal{F} = \{ f_a \}_{a \in \mathcal{A}}$ 
of strictly contractive self-maps of $\mathbb{R}^n$. It is well known that for any such IFS there exists a 
unique nonempty compact set $K$, called the \emph{attractor}, such that 
\begin{equation*}
K = \bigcup_{a \in \mathcal{A}} f_a(K). 
\end{equation*} 
When the maps are similarities, the set $K$ is called a \emph{self-similar set} 
(we will later modify the definition of self-similar sets slightly). 
The \emph{similarity dimension} 
of an IFS $\mathcal{F} = \{ f_a \}_{a \in \mathcal{A}}$ is the unique solution $d > 0$ of 
\begin{equation*}
\sum_{a \in \mathcal{A}} r_a^d = 1, 
\end{equation*}
where $r_a > 0$ is the contraction ratio of $f_a$. 
Recall that the IFS $\mathcal{F} = \{ f_a \}_{a \in \mathcal{A}}$ 
satisfies the \emph{open set condition (OSC)} if there exists a nonempty open set $O$ 
such that $f_a O \subset O$ for all $a \in \mathcal{A}$, and the images $f_a(O)$ are pairwise disjoint.  
%(we will later modify the definition of the OSC). 
It is well known that Hausdorff dimension and similarity dimension agree whenever OSC holds.

In this paper we consider projections of planar self-similar sets. 
Our interest in this paper is the case that projections have nonempty interior. 
Since projection does not increase Hausdorff dimension, we can restrict our consideration to the case that 
self-similar sets have Hausdorff dimension greater than $1$. 
In \cite{Marstrand}, Marstrand showed that for any planar Borel set with Hausdorff dimension greater than $1$, 
for a.e. directions the projections have positive Lebesgue measure. Recently, P. Shmerkin and B. Solomyak 
showed that for any planar self-similar set with Hausdorff dimension greater than $1$, 
the projections have positive Lebesgue measure except for the directions that have Hausdorff dimension $0$ \cite{SS}. 

However, very few results are known concerning the nonempty interior of projections of planar self-similar sets. 
It is natural to expect that for any planar self-similar set with Hausdorff dimension greater than $1$, 
one can create nonempty interior in the projections by 
applying arbitrary small perturbations to the generating contracting similarities, 
but this problem is known to be extremely hard. 
In this paper we consider weaker form of this question. 
We allow more freedom to the perturbation, 
and show that by applying this perturbation one can create nonempty interior 
in the projections. 
The basic idea of the proof is borrowed from \cite{Moreira3} and \cite{Takahashi}. 
In \cite{Takahashi}, by relying on the techniques invented by Moreira and Yoccoz in \cite{Moreira3}, 
the author showed that for any planar self-similar set 
one can create nonempty interior in a projection by applying arbitrary small perturbations,  
if the projection of a uniform self-similar measure has $L^2$-density and the generating contracting similarities 
are homotheties ($f: \mathbb{R}^2 \to \mathbb{R}^2$ is a homothety if 
$f(x) = r x + t$ for some $r \in (0, 1)$ and $t \in \mathbb{R}^2$).  
This paper can be considered as an extension of \cite{Takahashi}. 
Additional complication comes from the fact that the contracting similarities are no longer homotheties.

\subsection{Main results}\label{section2}

As explained in the introduction, self-similar set is normally defined as a set $K$ 
together with a set of contracting similarities that generates $K$. 
From below we modify the definition of self-similar sets to be the following: 

\begin{dfn}\label{selfsimilar}
A set $K \subset \mathbb{R}^2$ is \emph{self-similar} if the following holds: 
there exists a finite alphabet $\mathcal{A}$ and a set of contracting similarities 
$\mathcal{F} = \{ f_a \}_{a \in \mathcal{A}}$ 
on $\mathbb{R}^2$ 
such that 
\begin{equation*}
K = \displaystyle{ \bigcup_{a \in \mathcal{A}} f_a(K) }.
\end{equation*} 
\end{dfn}
Write $I = [0, 1]^2$. 
Without loss of generality, we can further assume that $K \subset I$. 
%Recall that the IFS $\mathcal{F} = \{ f_a \}_{a \in \mathcal{A}}$ 
%satisfies the \emph{open set condition (OSC)} if there exists a nonempty open set $O$ 
%such that $f_a O \subset O$ for all $a \in \mathcal{A}$, and the images $f_a(O)$ are pairwise disjoint.  

\begin{dfn}
Let $K$ be a self-similar set. 
We say that $K$ satisfies the \emph{OSC} 
if there exists a set of contracting similarities $\mathcal{F} = \{ f_a \}_{a \in \mathcal{A}}$ that generates $K$ 
and satisfies the OSC. 
%and a nonempty open set $O$ 
%such that $f_a O \subset O$ for all $a \in \mathcal{A}$, and the images $f_a(O)$ are pairwise disjoint. 
\end{dfn}

\begin{dfn}
Let $K, \widetilde{K}$ be self-similar sets. We say that $K$ and $\widetilde{K}$ are \emph{$\epsilon$-close} 
if the following holds: 
there exist 
sets of contracting similarities $\mathcal{F} = \{f_a\}_{a \in \mathcal{A}}$
 (resp. $\widetilde{\mathcal{F}} = \{\tilde{f}_{\tilde{a}}\}_{\tilde{a} \in \widetilde{ \mathcal{A} } }$) 
that generates $K$ (resp. $\widetilde{K}$) such that 
\begin{itemize}
\item[(i)] $\mathcal{A} = \widetilde{\mathcal{A}}$; 
\item[(ii)] $r_a  = \tilde{r}_a$ for all $a \in \mathcal{A}$;
\item[(iii)] $\| f_a - \tilde{f}_{a} \|_{I} / r_a < \epsilon$ for all $a \in \mathcal{A}$, 
\end{itemize}
where $\| f_a -\tilde{f}_a \|_{I} = \sup \{ | f_a(x) - \tilde{f}_a(x) | : x \in I \}$. % and $r_a$ (resp. $\tilde{r}_a$) is the 
%contraction ratio of $f_a$ (resp. $\tilde{f}_a$). 
\end{dfn}

Let $\ell_{\theta} \ni 0$ be the line that makes the angle $\theta + \pi/2$ with the $x$-axis, 
and let $\Pi_{\theta}$ be the projection onto $\ell_{\theta}$. 
Our main result is the following:

\begin{thm}\label{thm1}
Let $K$ be a self-similar set. Assume that $K$ satisfies the OSC and 
has Hausdorff dimension greater than $1$. 
Then, for every $\epsilon > 0$ there exists an open set $E \subset [0, \pi)$ and a self-similar set 
$\widetilde{K}$ such that 
\begin{itemize}
\item[(i)] $\left| [0, \pi) \setminus E \right| < \epsilon$; 
\item[(ii)] $\widetilde{K}$ is $\epsilon$-close to $K$; 
\item[(iii)] $\Pi_{\theta} \widetilde{K}$ contains an interval for all $\theta \in E$. 
\end{itemize}
\end{thm}

%Our result can be considered as an extension of \cite{Takahashi}. %Let us recall the main result of \cite{Takahashi}. 
%Let us recall the main result in \cite{Takahashi}. 

\begin{dfn}
A set $K \subset \mathbb{R}^2$ is \emph{h-self-similar} if the following holds: 
there exists a finite alphabet $\mathcal{A}$ and a set of contracting homotheties  
$\mathcal{F} = \{ f_a \}_{a \in \mathcal{A}}$ 
on $\mathbb{R}^2$ 
such that 
\begin{equation*}
K = \displaystyle{ \bigcup_{a \in \mathcal{A}} f_a(K) }.
\end{equation*} 
\end{dfn}
%\begin{dfn}
%Let $K$ be a h-self-similar set. We call a probability measure $\mu$ 
%a \emph{uniform self-similar measure} associated to $K$ if the following holds: 
%there exists a set of contracting homotheties $\mathcal{F} = \{ f_a \}_{a \in \mathcal{A}}$ that generates $K$ such that 
%\begin{equation*}
%\mu = \sum_{a \in \mathcal{A}} r_a^d f_a \mu, 
%\end{equation*}
%where $d = \dim_{S} (\mathcal{F})$ and $r_a > 0$ is the contraction ratio of $f_a$. 
%\end{dfn}
%Define uniform self-similar measures for h-self-similar sets analogously. 
In \cite{Takahashi}, the author proved the following: 
\begin{thm}\label{mukashi}
Let $K$ be a h-self-similar set. 
Assume that $K$ satisfies the OSC and has Hausdorff dimension greater than $1$.  
Then, for a.e. $\theta \in [0, \pi)$ we have the following: 
for every $\epsilon > 0$ there exists a h-self-similar set 
$\widetilde{K}$ such that 
\begin{itemize}
\item[(i)] $\widetilde{K}$ is $\epsilon$-close to $K$; 
\item[(ii)] $\Pi_{\theta} \widetilde{K}$ contains an interval.  
\end{itemize}
\end{thm}

Theorem \ref{thm1} is an extension of Theorem \ref{mukashi}. 
In Theorem \ref{thm1} the set $K$ is a general self-similar set (not necessarily h-self-similar), and 
furthermore, 
projections of the perturbed set have nonempty interior in ``most of the directions", while 
in Theorem \ref{mukashi} 
projection of the perturbed set has nonempty interior only in one particular direction.

\begin{rem}
Let $K$ be a h-self-similar set and $\{ f_a \}_{a \in \mathcal{A}}$ be a generating homotheties. 
Write $f_a(x) = r_a x + t_a$. 
Then it is easy to see that for any $\theta \in [0, \pi)$, $\Pi_{\theta} K$ is a self-similar set generated by 
$\{ r_a x + \Pi_{\theta} t_a \}_{a \in \mathcal{A}}$. 
Therefore, considering projections of h-self-similar sets is equivalent to considering one-dimensional self-similar sets. 
In \cite{Takahashi}, Theorem \ref{mukashi} is stated in terms of one-dimensional self-similar sets (with overlaps). 
If $K$ is not h-self-somilar then a projection of $K$ is not necessarily self-similar. 
\end{rem}

\begin{rem}
%Kaufman's proof of Marstrand theorem tells us that under the assumptions of Theorem \ref{mukashi} 
%projections of uniform self-similar measure have $L^2$-density for a.e. directions. 
The a.e. directions in Theorem \ref{mukashi} is the directions such that 
the projections of uniform self-similar measure have $L^2$-density. See section \ref{Marstrand}. 
\end{rem}

\subsection{Structure of the paper}
In section \ref{recurrence} we define renormalization operators and recurrent sets. 
The outline of the proof of Theorem \ref{thm1} is given in section \ref{outline}. 
In section \ref{construct} we will construct the set $\mathcal{L}$ 
which is the candidate of a recurrent set.  
The main difference of the proof between \cite{Takahashi} is in section \ref{construct}. 
In section \ref{key_prop} we will prove the key proposition, 
which roughly claims that with ``very high probability" any point in the set $\mathcal{L}$  
can return to itself by an action of a renormalization operator.

\section{Projections of planar self-similar sets and recurrent sets}\label{recurrence}

\subsection{Renormalization}
For any $\theta \in [0, \pi)$, we parametrize the line $\ell_{\theta}$ by 
\begin{equation*}
\begin{aligned}
\mathbb{R} &\to \ell_{\theta} \\
t &\mapsto t 
\begin{pmatrix}
-\sin \theta \\
\cos \theta 
\end{pmatrix}. 
\end{aligned}
\end{equation*}
We use this identification freely. 
Throughout this section, we fix a self-similar set $K \subset I$   
and a set of contracting similarities $\mathcal{F} = \{ f_a \}_{a \in \mathcal{A}}$ that generates $K$.  
Let $Q$ be the set of all lines in $\mathbb{R}^2$. 
For $u \in Q$, we denote by $\arg u \in [0, \pi)$ the angle that $u$ makes with the $x$-axis. 
Write $u \cap \ell_{\arg u} \in \mathbb{R}$ by $\langle u \rangle$.  
It is easy to see that the map 
\begin{equation}\label{identification0}
\begin{aligned}
Q &\to [0, \pi) \times \mathbb{R} \\
u &\mapsto ( \arg u, \langle u \rangle )
\end{aligned}
\end{equation}
is a bijection. From below we use this identification freely. 
For $a \in \mathcal{A}$, we define a \emph{renormalization operator} $T_{a} (\cdot): Q \to Q$ by 
\begin{equation}\label{renormalization}
T_{a} ( u ) = f^{-1}_a( u ). 
\end{equation}
Similarly, for $a_1, a_2 \in \mathcal{A}$ we define a map $T_{a_1 a_2} (\cdot): Q \to Q$ by 
\begin{equation*}
T_{a_1 a_2} ( u ) = f^{-1}_{a_2} \circ f^{-1}_{a_1}( u ), 
\end{equation*}
and call this also a renormalization operator. 
Note that we have $T_{a_1 a_2} = T_{a_2} \circ T_{a_1}$.

\begin{rem}
It is convenient to see $u \in Q$ as the ``relative position" between $I$ and $u$. 
For any $a \in \mathcal{A}$, $T_a(u)$ can be seen as the ``relative position" between 
the square $f_a(I)$ and $u$. 
\end{rem}

\subsection{Recurrent sets}

We say that $u \in Q$ is \emph{intersecting} if $K \cap u \neq \emptyset$.  
 
\begin{lem}\label{crucial}
Let $u \in Q$. Then $u$ is intersecting if and only if the following holds: 
there exists $M > 0$ and $a_i \in \mathcal{A} \ (i = 1, 2, \cdots)$ such that the sequence 
$\{ u_i \} \ (i = 0, 1, \cdots)$ defined by 
\begin{equation}\label{u}
u_0 = u, \ u_i = T_{a_i} u_{i-1} 
\end{equation} 
satisfies 
$| \langle u_i \rangle  | < M \ (i = 0, 1, 2, \cdots)$. 
\end{lem}

\begin{proof}
Assume first that $u$ is intersecting. 
Let $x \in K$ be such that $x \in u$, and let 
$a_i \in \mathcal{A}  \ (i = 1, 2, \cdots)$ be a sequence such that 
\begin{equation*}
x = \bigcap_{i=1}^{\infty} f_{a_1} \circ \cdots \circ f_{a_i} (I). 
\end{equation*} 
Note that  
\begin{equation}\label{1}
f^{-1}_{a_i} \circ \cdots \circ f^{-1}_{a_1} (x) \in I 
\end{equation} 
for all $i \in \mathbb{N}$.
Define $\{ u_i \}$ by (\ref{u}). Then we have 
\begin{equation}\label{2}
u_i = f^{-1}_{a_i} \circ \cdots \circ f^{-1}_{a_1} (u).
\end{equation} 
Since $x \in u$, by (\ref{1}) and (\ref{2}) we conclude that $u_i \cap I \neq \emptyset$. This implies that 
$| \langle u_i \rangle  | \leq 1$. 

Assume next that $u$ is not intersecting. Let us take $a_i \in \mathcal{A} \ (i = 1, 2, \cdots)$, 
and let $\{ u_i \}$ be the sequence 
defined by (\ref{u}). 
Write 
\begin{equation*}
x = \bigcap_{i=1}^{\infty} f_{a_1} \circ \cdots \circ f_{a_i} (I). 
\end{equation*}
Since $x \notin u$, we have $f_{a_1} \circ \cdots \circ f_{a_i}(I) \cap u = \emptyset$ for sufficiently large $i$. 
Since the size of the square $f_{a_1} \circ \cdots \circ f_{a_i} (I)$ goes to $0$, 
we obtain $\lim_{i \to \infty} | \langle u_i \rangle  | = \infty$. 
\end{proof}

The above lemma leads to the following definition: 

\begin{dfn}
We call a nonempty bounded set $\mathcal{L} \subset Q$ a \emph{recurrent set} 
if for every $u \in \mathcal{L}$, 
there exists $a \in \mathcal{A}$ 
such that $T_{a} u \in \mathcal{L}$.  
\end{dfn}

Lemma \ref{crucial} implies the following: 

\begin{prop}
Let $\mathcal{L}$ be a recurrent set and $\theta \in [0, \pi)$. 
If the set $\{ t : (\theta, t) \in \mathcal{L}  \}$ 
contains an interval, then $\Pi_{ \theta } K$ contains an interval. 
\end{prop}

\section{Outline of the proof of the main theorem}\label{outline}

\subsection{Perturbation}
In this section, we discuss the outline of the proof of Theorem \ref{thm1}. 
Let $\epsilon > 0$. 
Let $K \subset I$ be a self-similar set, and let 
$\mathcal{F} = \{ f_a \}_{a \in \mathcal{A}}$ be a set of 
contracting similarities that generates $K$. Denote the Hausdorff dimension of $K$ by $d$. 
Let $\mu$ be the associated uniform self-similar measure, i.e., 
$\mu$ is the unique Borel probability measure such that 
\begin{equation*}
\mu = \sum_{a \in \mathcal{A}} r_a^d f_a \mu, 
\end{equation*}
where $f_a \mu$ is the push-forward of $\mu$ under the map $f_a$. 
By retaking $\mathcal{F}$ if necessary, we can further assume that 
\begin{equation*}
c_0^{-1} \rho^{1/2} < r_a < c_0 \rho^{1/2}
\end{equation*}
for 
sufficiently small $\rho > 0$. 
Let $\mathcal{A}_{1}, \mathcal{A}_2$ be such that 
$\mathcal{A}_1 \sqcup \mathcal{A}_2 = \mathcal{A}$ and $|\mathcal{A}_1| = | \mathcal{A}_2 | = | \mathcal{A} | / 2$. 
Let $c_1 > 0$ be a sufficiently large constant, to be chosen later. 

\begin{rem}
In the proof we use constants $c_{k} \ (k = 0, 1, \cdots, 10)$. 
They may depend on each other but can be taken independently of $\rho > 0$. 
\end{rem}

Let $a \in \mathcal{A}_1$ and $\omega \in (-\epsilon, \epsilon) \times (-1, 1)^2$. Write 
$\omega = (\varphi, \gamma)$, where $\varphi \in (-\epsilon, \epsilon)$ and $\gamma \in (-1, 1)^{2}$. 
Let $f_a^{\omega}$ be the contracting similarity which satisfies the following: 
$f^{ \omega }_a(I)$ is the square that is obtained by rotating the square 
$f_a(I)$ by the angle $\varphi \in (-\epsilon, \epsilon)$ and shifted by 
$\gamma c_1 \rho \in 
(-c_1 \rho, c_1 \rho)^2$. 

Define 
\begin{equation*}
\Omega = \left(  (-\epsilon, \epsilon) \times (-1, 1)^2  \right)^{\mathcal{A}_1}. 
\end{equation*}
Let $\underline{\omega} = ( \omega_a )_{a \in \mathcal{A}_1} \in \Omega$, and 
denote $\omega_a = ( \varphi_a, \gamma_a )$, 
where $\varphi_a \in (-\epsilon, \epsilon)$ and $\gamma_a \in (-1, 1)^2$. 
We define 
\begin{equation*}
\mathcal{F}^{\underline{\omega}} = \{ f^{ \underline{\omega} }_a \}_{a \in \mathcal{A}}, 
\end{equation*}
a set of contracting similarities, in the following way: 
\begin{equation*}
f^{\underline{\omega}}_a = 
\begin{cases}
f^{\omega_a}_a & \text{ if \ } a \in \mathcal{A}_1 \\
f_a & \text{ if \ } a \in  \mathcal{A}_2. 
\end{cases}
\end{equation*}
Let $K^{\underline{\omega}}$ 
be the self-similar set generated by $\mathcal{F}^{\underline{\omega}}$. 
Note that if $\rho > 0$ is sufficiently small, then $K^{\underline{\omega}}$ is $\epsilon$-close to $K$. 

Recall that we defined the renormalization operator in (\ref{renormalization}). 
For $a \in \mathcal{A}_1$ and $\omega \in (-\epsilon, \epsilon) \times (-1, 1)^2$, we define the 
renormalization operator $T^{\omega}_a$ in analogous way. 
For $\underline{\omega} = ( \omega_a )_{a \in \mathcal{A}_1}\in \Omega$, 
we define 
\begin{equation*}
T^{\underline{\omega}}_a = 
\begin{cases}
T^{\omega_a}_a & \text{ if \ } a \in \mathcal{A}_1 \\
T_a & \text{ if \ } a \in \mathcal{A}_2. 
\end{cases}
\end{equation*}

\subsection{Outline of the proof}
In section \ref{construct}, we will construct the set 
$E \subset [0, \pi)$, and the set $L(\theta) \subset (-1, 1)$ for all $\theta \in E$. 
Define 
\begin{equation*}
\mathcal{L}^0 = \left\{ (\theta, t) : \theta \in E, \, t \in L( \theta ) \right\}. 
\end{equation*}  
Let 
\begin{equation*}
\mathcal{L}^{1} = 
\left\{  (\theta, t) : \exists (\theta_0, t_0) \in \mathcal{L}^0 \text{ with } | \theta - \theta_0 | < \rho, | t - t_0 | < \rho \right\}
\end{equation*}
and 
\begin{equation*}
\mathcal{L} = 
\left\{  (\theta, t) : \exists (\theta_0, t_0) \in \mathcal{L}^0 \text{ with } | \theta - \theta_0 | < \rho / 2, | t - t_0 | < \rho / 2 \right\}. 
\end{equation*}
We show that $\mathcal{L}$ is a recurrent set for some $\underline{\omega} \in \Omega$.  
For $u \in \mathcal{L}^1$, we define $\Omega^0(u) \subset \Omega$ to be the set of all 
$\underline{\omega} \in \Omega$ such that the following holds: 
there exists $\underline{b} \in \mathcal{A}^2$ and the image 
\begin{equation*}
T^{\underline{\omega}}_{ \underline{b} } (u) = \hat{u}
\end{equation*}
satisfies $\hat{u} \in \mathcal{L}^0$. 
The following crucial estimate will be proven in section \ref{key_prop}. 

\begin{prop}\label{key_lem}
There exists $c_2 > 0$ such that for any $u \in \mathcal{L}^1$, 
\begin{equation*}
\mathbb{P} \left( \Omega \setminus \Omega^0(u) \right) \leq 
\exp \left( -c_2 \rho^{ -\frac{1}{2} (d - 1) } \right). 
\end{equation*}
\end{prop}

%\begin{rem}
The sets $E$ and $L(\theta)$ are 
constructed in such a way that Proposition \ref{key_lem} holds. 
Below we prove Theorem \ref{thm1} assuming 
Proposition \ref{key_lem}. 
In section \ref{construct} we construct $E$ and $L(\theta)$, 
and show that the measure of the set $L(\theta)$ is bounded away from 
zero uniformly. 
Combining all these properties we prove Proposition \ref{key_lem} in section \ref{key_prop}. \vspace{2mm}
%\end{rem}

\begin{centering}
\begin{figure}[t]
\includegraphics[scale=0.88]{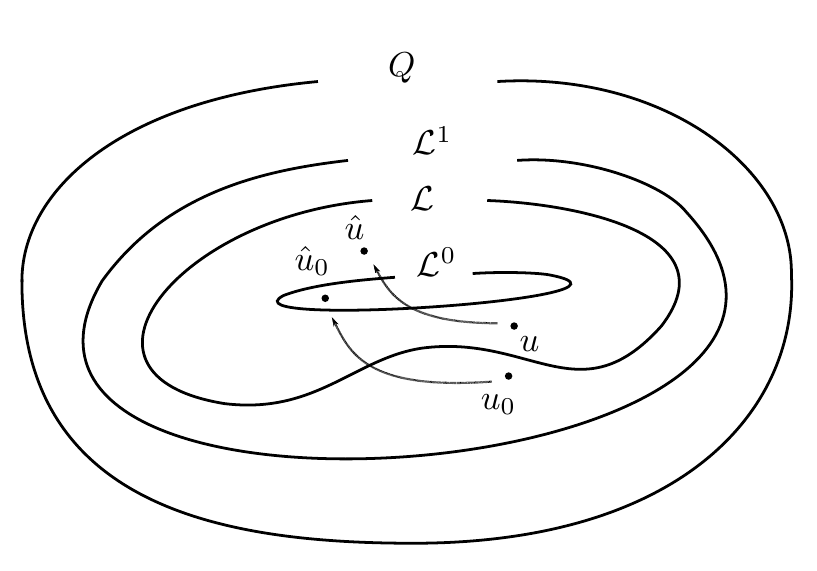}
\caption{Recurrent set}
\label{recurrence_figure}
\end{figure}
\end{centering}

We choose a finite $\rho^{ 5/2 }$-dense subset $\Delta$ of $\mathcal{L}^1$. 
Note that 
\begin{equation*}
| \Delta | \leq c_3 \rho^{ -5/2 } \cdot \rho^{ -5/2 } = c_3 \rho^{-5}. 
\end{equation*}
Now, if $\rho > 0$ is small enough, 
\begin{equation*}
c_{3} \, \rho^{-5} \exp \left( -c_2 \rho^{ -\frac{1}{2} ( d - 1 ) } \right) < 1, 
\end{equation*}
and therefore we can find $\underline{\omega}_0 \in \Omega$ such that 
$\underline{\omega}_0 \in \Omega^0(u)$ for all $u \in \Delta$.

\begin{rem}
The above is saying that any $u \in \Delta$ can return to $\mathcal{L}^0$ by an 
action of the renormalization operator of the form 
$T^{\underline{\omega}_0}_{ \underline{b} }$. 
\end{rem}

Theorem \ref{thm1} follows from the following claim: 

\begin{claim}
For $\underline{\omega}_0 \in \Omega$, the set $\mathcal{L}$ is a recurrent set.  
\end{claim}

\begin{proof}[proof of the claim]
Let $u \in \mathcal{L}$.  
Let $u_0 \in \Delta$ be such that 
$| \arg u - \arg u_0 | < \rho^{5/2}$ and $| \langle u \rangle - \langle u_0 \rangle | < \rho^{5/2}$. 
By the choice of $\underline{\omega}_0$, we have $\underline{\omega}_0 \in \Omega^{0}( u_0 )$.  
Therefore, there exists $\underline{b} \in \mathcal{A}^2$ such that, writing 
\begin{equation*}
T_{\underline{b}}^{\underline{\omega}_0} (u_0) = \hat{u}_0, 
\end{equation*}
we have $\hat{u}_0  \in \mathcal{L}^0$. Let 
\begin{equation*}
T_{\underline{b}}^{\underline{\omega}_0} (u) = \hat{u}.
\end{equation*}
It is easy to see that $| \arg \hat{u} - \arg \hat{u}_0 | < \rho^{5/2}$ and 
$| \langle \hat{u} \rangle - \langle \hat{u}_0 \rangle |$ 
is of order $\rho^{3/2}$. Therefore, we obtain $\hat{u} \in \mathcal{L}$. 
\end{proof}

\section{Construction of the set $E$ and $L(\theta)$}\label{construct}

\subsection{Construction of $E$}\label{Marstrand}
Kaufman's proof of Marstrand's theorem tell us that 
the measure $\Pi_{\theta} \mu$ is 
absolutely continuous with respect to Lebesgue measure 
for a.e. $\theta \in [0, \pi)$, with $L^2$-density $\chi_{\theta}$ 
satisfying 
\begin{equation*}
\int_{[0, \pi)} \| \chi_{\theta} \|_{L^2}^2 \, d \theta < c_{4}. 
\end{equation*}
See, for example, \cite{PalisTakens}. 
We define 
\begin{equation*}
E = \left\{ \theta \in [0, \pi) : \| \chi_{\theta} \|_{L^2}^2 < c_5 \right\}, 
\end{equation*}
where $c_5 > 0$ is a sufficiently large constant so that 
$| [0, \pi) \setminus E | < \epsilon / 2$.

\subsection{Construction of $L(\theta)$}
Let $a_1 \in \mathcal{A}_1, a_2 \in \mathcal{A}_2$ 
and $\omega = (\varphi, \gamma) \in (-\epsilon, \epsilon) \times (-1, 1)^2$.  
We denote the square $f^{\omega }_{a_1}(I)$ by $I^{\varphi, \gamma}(a_1)$, and the square 
$f^{\omega }_{a_1} \circ f_{a_2} (I)$ by $I^{ \varphi, \gamma }( a_1 a_2 )$. 
Let $u = ( \theta, t ) \in Q$. 
For 
\begin{equation*}
( \hat{\theta}, \hat{t} ) = T_{a_1}^{\omega}(u), 
\end{equation*}
we denote  
$\hat{\theta}$ by $\arg_{\theta} I^{ \varphi, \gamma }( a_1 )$ and 
$\hat{t}$ by $\pos_{(\theta, t)} I^{ \varphi, \gamma } (a_1 )$. 
Define $\arg_{\theta} I^{ \varphi, \gamma }( a_1 a_2 )$ and 
$\pos_{(\theta, t)} I^{ \varphi, \gamma } (a_1 a_2)$ analogously.

With $c_{6} > 0$ conveniently small, to be chosen later, let 
\begin{equation*}
N = c^{2}_{6} \rho^{-\frac{1}{2} ( d - 1 ) }. 
\end{equation*}
Let $\theta \in E$. 
We define $L(\theta)$ to be the set of points 
$t \in (-1, 1)$ such that the following holds 
($c_7 > 0$ is a sufficiently small constant to be chosen later) :  
there exist mutually distinct words 
$\{ a_1^1, a_1^2, \cdots, a_1^N \} \subset \mathcal{A}_1$ and the sets 
$\Phi_{i} \subset (-\epsilon, \epsilon)$ with $| \Phi_{i} | > c_{7}$ ($i = 1, 2, \cdots, N$) 
such that for all $a^i_1$ and $\varphi \in \Phi_{i}$, there exists $a^{i}_2 \in \mathcal{A}_2$ such that 
\begin{equation*}
\arg_{\theta} I^{ \varphi, 0 }(a_1^i a_2^{i} ) \in E \text{ \ and \ } 
| \pos_{(\theta, t)} I^{ \varphi, 0 }(a_1^i a_2^{i} ) | \leq 1. 
\end{equation*}
In the next section, we will prove the following estimate:
\begin{prop}\label{nice_estimate}
If $c_{6} > 0$ is sufficiently small, there exists $c_{8} > 0$ such that $| L(\theta) | > c_{8}$ for all $\theta \in E$.  
\end{prop}

\subsection{Projections of the squares $I(a)$}
Let $\theta \in E$, and let $\widetilde{ \mathcal{A} } \subset \mathcal{A}$ be such that 
$| \widetilde{ \mathcal{A} } | > | \mathcal{A} | / 8$. 
For $a \in \widetilde{ \mathcal{A} }$, we have  
\begin{equation}\label{measure}
c_{9}^{-1} \rho^{\frac{1}{2}d} < \mu ( I(a) ) < 
c_{9} \rho^{\frac{1}{2} d}. 
\end{equation}
Write $\mathcal{J}( a ) := \Pi_{\theta} ( I( a ) )$.
Then 
\begin{equation*}
c_{9}^{-1} \rho^{1/2} < | \mathcal{J}(a) | < c_{9} \rho^{1/2}. 
\end{equation*}
We call $a \in \widetilde{ \mathcal{A} }$ \emph{$(\widetilde{\mathcal{A}}, \theta)$-good} if there are no more than 
$c_{6}^{-1} \rho^{-\frac{1}{2}( d - 1 )}$ intervals $\mathcal{J}( \tilde{a} )$ $(\tilde{a} \in \widetilde{A})$ 
whose centers are distant from the center of $\mathcal{J}(a)$ by less than $c_9^{-1} \rho^{ 1/2 }$. 
Call $a \in \widetilde{ \mathcal{A} }$ \emph{$( \widetilde{\mathcal{A}}, \theta )$-bad} if it is not 
$(\widetilde{\mathcal{A}}, \theta)$-good. 
Recall that, since $\theta \in E$, the measure $\Pi_{\theta} \mu$ has $L^2$-density $\chi_{\theta}$ which satisfies  
$\left\| \chi_{\theta} \right\|^2_{L^2} < c_5$.

\begin{lem}\label{bad_good}
The number of $(\widetilde{ \mathcal{A} }, \theta)$-bad words is less than 
\begin{equation*}
6 c_5 c_{6}  c^3_{9} \rho^{-\frac{1}{2}d}. 
\end{equation*}
In particular, if $c_6 > 0$ is sufficiently small, the number of $(\widetilde{\mathcal{A}}, \theta)$-good words is at least 
$| \widetilde{\mathcal{A}} | / 2$. 
\end{lem}

\begin{proof} 
Let $a \in \widetilde{ \mathcal{A} }$ be $(\widetilde{ \mathcal{A} }, \theta)$-bad. 
Then we have 
\begin{equation*}
\begin{aligned}
\int_{3 \mathcal{J}( a )} \chi_{\theta} 
&\geq c^{-1}_{9} \rho^{ \frac{1}{2}d} \cdot c_{6}^{-1} \rho^{ -\frac{1}{2} (d -1) } \\
&= c^{-1}_{6} c^{-1}_{9} \rho^{1/2} > \frac{1}{3} c_{6}^{-1} c_{9}^{-2} | 3 \mathcal{J}( a ) |, 
\end{aligned}
\end{equation*}
where $3 \mathcal{J}(a)$ is the interval of the same center as $\mathcal{J}(a)$ and length $3| \mathcal{J}(a) |$. 
By the Cauchy-Schwarz inequality, 
\begin{equation*}
\begin{aligned}
\frac{1}{3} c_{6}^{-1} c_{9}^{-2} | 3 \mathcal{J}( a ) | \int_{3 \mathcal{J}( a )} \chi_{\theta}  &\leq 
\left( \int_{3 \mathcal{J}( a )} \chi_{\theta}  \right)^{2} \\
&\leq | 3 \mathcal{J} ( a ) | \int_{3 \mathcal{J}( a )} \chi_{\theta}^{2}, 
\end{aligned}
\end{equation*}
and thus 
\begin{equation*}
\int_{3 \mathcal{J}( a )} \chi_{\theta}^{2} \geq 
\frac{1}{3} c_{6}^{-1} c_{9}^{-2}  \int_{3 \mathcal{J}( a )} \chi_{\theta}. 
\end{equation*}
Let $\mathcal{J}^{*}$ be the union over all bad words 
$a \in \widetilde{ \mathcal{A} }$ of the intervals $3\mathcal{J}( a )$. One can 
extract a subfamily of intervals whose union is $\mathcal{J}^{*}$ and does not cover any point more than twice. 
Then we obtain 
\begin{equation*}
\int_{\mathcal{J}^*} \chi_{\theta}^{2}  \geq 
\frac{1}{6} c_{6}^{-1} c_{9}^{-2}  \int_{ \mathcal{J}^{*} } \chi_{\theta}.
\end{equation*}
Therefore,   
\begin{equation*}
 \int_{ \mathcal{J}^{*}} \chi_{\theta} \leq 6 c_5 c_{6} c_{9}^2. 
\end{equation*}
As $\mathcal{J}^{*}$ contains $\mathcal{J}( a )$ 
for all bad $a \in \widetilde{ \mathcal{A} }$, together with (\ref{measure}) 
the estimate of the lemma follows. 
\end{proof}

Lemma \ref{bad_good} implies the following: 
 
\begin{lem}\label{proj}
\begin{equation*}
\Big| \bigcup_{ a \in \widetilde{\mathcal{A}} } \mathcal{J}(a) \Big| > c_6 c_{10}. 
\end{equation*}
\end{lem} 
 
\begin{proof}
We have 
\begin{equation*}
\begin{aligned}
\Big| \bigcup_{ a \in \widetilde{\mathcal{A}} } \mathcal{J}(a) \Big| &> 
\frac{1}{2} | \widetilde{ \mathcal{A} } | \cdot \left( c^{-1}_6 \rho^{-\frac{1}{2}(d  - 1)} \right)^{-1} \cdot c_9^{-1} \rho^{1/2} \\  
&> \frac{1}{2} \cdot \frac{1}{8} c_9^{-1} \rho^{-\frac{1}{2} d} \cdot c_6 \rho^{\frac{1}{2}(d  - 1)} \cdot c_9^{-1} \rho^{1/2} 
= \frac{1}{16} c_6 c^{-2}_{9}. 
\end{aligned}
\end{equation*}
\end{proof}

\subsection{Proof of Proposition \ref{nice_estimate}} 
We fix $\theta \in E$ for the rest of the section. 
For $a_1 \in \mathcal{A}_1$ and $\varphi \in (-\epsilon, \epsilon)$, 
define 
\begin{equation*}
\Lambda_{a_1, \varphi} = 
\left\{ a_2 \in \mathcal{A}_2 : \arg_{\theta} I^{ \varphi, 0 } (a_1 a_2)  \in E \right\} 
\end{equation*}
and 
\begin{equation*}
\Phi^{*}_{a_1} = \left\{ \varphi \in (-\epsilon, \epsilon) : | \Lambda_{a_1, \varphi} | > | \mathcal{A}_2 | / 4 \right\}. 
\end{equation*}
The following lemma is immediate. 
\begin{lem}
For any $a_1 \in \mathcal{A}_1$, 
we have $| \Phi^{*}_{a_1} | > \epsilon$. 
\end{lem}

\begin{proof}
By the construction of $E$, we have 
\begin{equation*}
\left| \{ \varphi \in (-\epsilon, \epsilon) : \arg_{\theta} I^{ \varphi, 0 }(a_1 a_2) \in E \} \right| > \frac{3}{2} \epsilon 
\end{equation*}
for all $a_2 \in \mathcal{A}_2$. 
Let $\psi$ be the sum, over $a_2 \in \mathcal{A}_2$, 
of the characteristic functions of 
$\{ \varphi \in (-\epsilon, \epsilon) : \arg I_{\theta}^{\varphi, 0}(a_1 a_2) \in E \}$. 
Note that $0 \leq \psi \leq | \mathcal{A}_2 |$ and 
\begin{equation*}
\varphi \in \Psi^{*}_{a_1} \iff \psi( \varphi ) > | \mathcal{A}_2 | / 4.
\end{equation*}   
Therefore, we have
\begin{equation*}
\begin{aligned}
\frac{3}{2} \epsilon \cdot | \mathcal{A}_2 | < 
\int_{(-\epsilon, \epsilon)} \psi &= \int_{\Phi^{*}_{a_1}} \psi + \int_{( -\epsilon, \epsilon ) \setminus \Phi^{*}_{a_1}} \psi \\
&< | \mathcal{A}_2 | | \Phi^{*}_{a_1} | + \frac{ | \mathcal{A}_2 | }{4} \cdot 2 \epsilon. 
\end{aligned}
\end{equation*}
The claim follows from this. 
\end{proof}

For $a_1 \in \mathcal{A}_1$, let  
\begin{equation*}
\Phi^{**}_{a_1} = \Phi^{*}_{a_1} \cap \{ \varphi \in (-\epsilon, \epsilon) : \arg_{\theta} I^{\varphi, 0}(a_1) \in E \}. 
\end{equation*}
By the above lemma, 
we have $| \Phi^{**}_{a_1} | > \epsilon / 2$. 
For $a_1 \in \mathcal{A}_1, a_2 \in \mathcal{A}_2$ and $\varphi \in \Phi^{**}_{a_1}$, we denote  
\begin{equation*}
J^{\varphi}( a_1 a_2 ) = \Pi_{\theta}( I^{ \varphi, 0 }(a_1 a_2) ), 
\end{equation*}
and for $a_1 \in \mathcal{A}_1$ and $\varphi \in \Phi^{**}_{a_1}$, write 
\begin{equation*}
J^{\varphi}(a_1) = \bigcup_{a_2 \in \Lambda_{a_1, \varphi} } J^{\varphi}( a_1 a_2 ). 
\end{equation*}
By Lemma \ref{proj}, we have 
\begin{equation*}
| J^{\varphi}(a_1) | > c_6 c_{10} \rho^{1/2}. 
\end{equation*} 

For $t \in \mathbb{R}$, let 
\begin{equation*}
\Phi_{a_1, t} = \left\{ \varphi \in \Phi_{a_1}^{**} : t \in J^{\varphi}(a_1) \right\}. 
\end{equation*}
Write 
\begin{equation*}
J_{a_1} = \left\{ t \in \mathbb{R} : \left| \Phi_{a_1, t} \right| > \frac{1}{2} c_6 c_9^{-1} c_{10} | \Phi^{**}_{a_1} | \right\}. 
\end{equation*}
Note that we have 
\begin{equation*}
\begin{aligned}
\frac{1}{2} c_{6} c_9^{-1} c_{10} | \Phi_{a_1}^{**} | &> \frac{1}{2} c_{6} c_9^{-1} c_{10} \cdot \frac{1}{2} \epsilon \\
&= c_7. 
\end{aligned}
\end{equation*}

\begin{lem}
We have 
\begin{equation*}
| J_{a_1} | > \frac{1}{2} c_6 c_{10} \rho^{1/2}.
\end{equation*} 
\end{lem}

\begin{proof}
Let us integrate the characteristic function of 
\begin{equation*}
\left\{ (t, \varphi) : t \in J^{\varphi}(a_1) \text{ for some } \varphi \in \Phi^{**}_{a_1} \right\}
\end{equation*}
over $\{ (t, \varphi) : t \in \mathbb{R}, \varphi \in ( -\epsilon, \epsilon ) \}$. By Fubini's theorem, we have
\begin{equation*}
\begin{aligned}
| \Phi^{**}_{a_1} | \cdot c_6 c_{10} \rho^{1/2} &< \int_{\varphi \in (-\epsilon, \epsilon)} \int_{t \in \mathbb{R}} \\ 
&= \int_{t \in \mathbb{R}} \int_{\varphi \in (-\epsilon, \epsilon)} \\ 
&= \int_{t \in J_{a_1}} \int_{\varphi \in (-\epsilon, \epsilon)}  + 
\int_{t \in \mathbb{R} \setminus J_{a_1} }  \int_{\varphi \in (-\epsilon, \epsilon) } \\
&< | J_{a_1} | | \Phi^{**}_{a_1} | + c_9 \rho^{1/2} \cdot \frac{1}{2} c_{6} c_{9}^{-1} c_{10} | \Phi^{**}_{a_i} |. 
\end{aligned}
\end{equation*}
The claim follows from this. 
\end{proof}

Let $\psi$ be the sum, over $( \mathcal{A}_1, \theta )$-good words 
$a_1 \in \mathcal{A}_1$, of the characteristic functions of 
$J_{a_1}$. Note that $\supp \psi \subset ( -1, 1)$ and  
\begin{equation*}
0 \leq \psi \leq c_{6}^{-1} \rho^{ -\frac{1}{2} (d - 1)}. 
\end{equation*}
Let $\mathcal{D} = \{  \psi \geq  c^2_{6} \rho^{ -\frac{1}{2} (d - 1)}  \}$.
Then we have 
\begin{equation*}
\begin{aligned}
| \mathcal{D} | \cdot c_{6}^{-1} \rho^{ -\frac{1}{2} (d - 1)} + 
2 \cdot  c^2_{6} \rho^{ -\frac{1}{2} (d - 1)} 
&\geq \int_{L} \psi + \int_{ (-1, 1) \setminus L } \psi  \\
&= \int \psi \\
&\geq \frac{1}{2} c_6 c_{10} \rho^{1/2} \cdot \frac{1}{2} \cdot \frac{1}{2} \cdot c_{9}^{-1} \rho^{-\frac{1}{2} d}. 
\end{aligned}
\end{equation*}
Take $c_6 > 0$ small enough so that 
\begin{equation*}
2 c_6^2 < \frac{1}{16} c_{6} c_{9}^{-1} c_{10}. 
\end{equation*}
holds. 
Then we obtain 
\begin{equation*}
| \mathcal{D} | \geq \frac{1}{16} c_{6} c_{9}^{-1} c_{10} \cdot c_{6} =: c_{8}. 
\end{equation*}
Since $\mathcal{D} \subset L(\theta)$, we have proved that 
\begin{equation*}
| L(\theta) | \geq c_{8}. 
\end{equation*}  
This concludes the proof of Proposition \ref{nice_estimate}.

\section{Proof of the key Proposition}\label{key_prop}

\subsection{Proof of Proposition \ref{key_lem}} 
In this section, we prove Proposition \ref{key_lem}. 
Fix $(\theta, t) \in \mathcal{L}^1$. 
Let $( \theta', t' ) \in \mathcal{L}^0$ be such that $| \theta - \theta' | < \rho$ and 
$| t - t' | < \rho$. 

Then, there exist mutually distinct words 
$\{ a_1^1, a_1^2, \cdots, a_1^N \} \subset \mathcal{A}_1$ and the sets 
$\Phi_{i} \subset (-\epsilon, \epsilon)$ with $| \Phi_{i} | > c_{7}$ ($i = 1, 2, \cdots, N$) 
such that for all $a^i_1$ and $\varphi' \in \Phi_{i}$, there exists $a_2^{i} \in \mathcal{A}_2$ such that 
\begin{equation*}
\arg_{ \theta' } I^{ \varphi', 0 }(a_1^i a_2^{i}) \in E \text{ \ and \ } 
| \pos_{( \theta', t' )} I^{ \varphi', 0 }(a_1^i a_2^{i}) | \leq 1. 
\end{equation*}

Let 
\begin{equation*}
\mathcal{A}_3 = \left\{ a^1_1, a^2_1, \cdots, a^N_1 \right\}.
\end{equation*}
Write 
\begin{equation*}
\begin{aligned}
\Omega &= \left( ( -\epsilon, \epsilon ) \times (-1, 1)^2 \right)^{ \mathcal{A}_3 } 
\times \left( ( -\epsilon, \epsilon ) \times (-1, 1)^2 \right)^{ \mathcal{A}_1 \setminus \mathcal{A}_3 } ,  \\
\underline{\omega} &= ( \underline{\omega}', \underline{\omega}'' ), \text{ and \ } 
 \underline{\omega}' = \left( \omega_1, \omega_2, \cdots, \omega_N \right). 
\end{aligned}
\end{equation*}
Denote $\omega_i = ( \varphi_i, r_i )$, where $\varphi_i \in (-\epsilon, \epsilon)$ and $r_i \in (-1, 1)^2$. 
By Fubini's Theorem, Proposition \ref{key_lem} follows from the following claim:

\begin{claim}
There exists $c'_2 > 0$ such that for any $a_1^i \in \mathcal{A}_3$, 
\begin{equation*}
\left|  \left\{  \omega_i : \exists a_2^i \in \mathcal{A}_2 \text{ such that } T_{a_1^i a_2^i}^{\omega_i}(\theta, t) \in \mathcal{L}^0 
 \right\}  \right| > c'_2. 
\end{equation*}
\end{claim}

\begin{proof}[Proof of the claim]
Let $a_1^i \in \mathcal{A}_3$ and $\varphi' \in \Phi_i$. 
Let $\varphi = \varphi' + \theta - \theta'$. 
%write $\Phi_i = \left( \Phi_i' + \theta - \theta' \right) \cap (-\epsilon, \epsilon)$. 
Let $a_2^i \in \mathcal{A}_2$ be such that 
\begin{equation*}
\arg_{ \theta' } I^{ \varphi', 0 }(a_1^i a_2^i) \in E \text{ \ and \ } 
| \pos_{( \theta', t' )} I^{ \varphi', 0 }(a_1^i a_2^{i}) | \leq 1. 
\end{equation*}
Therefore, 
\begin{equation*}
\arg_{\theta} I^{\varphi, \gamma} (a^{i}_1 a^{i}_2) \in E
\end{equation*}
for all $\gamma \in (-1, 1)^2$. It is easy to see that 
\begin{equation*}
\big| \pos_{( \theta, t )} I^{\varphi, 0} (a_1^{i} a_2^{i} ) - 
\pos_{( \theta' t' )} I^{\varphi', 0} ( a_1^i a_2^{i} )  \big| 
\end{equation*}
is of order $1$. 
Therefore, 
\begin{equation*}
\big| \pos_{( \theta, t )} I^{\varphi, 0} ( a_1^i a_2^i ) \big|
\end{equation*}
 is also of order $1$. 
It follows that, by taking $c_1 > 0$ large enough, we obtain 
\begin{equation*}
\left| \left\{ \gamma \in (-1, 1)^2 : \pos_{(\theta, t)} I^{\varphi, \gamma } ( a_1^{i} a_{2}^{i} ) 
\in L(\theta)  \right\} \right| > c_2''. 
\end{equation*}
Therefore, 
\begin{equation*}
\begin{aligned}
\left|  \left\{  \omega_i : \exists a_2^i \in \mathcal{A}_2 \text{ such that } T_{a_1^i a_2^i}^{\omega_i}(\theta, t) \in \mathcal{L}^0 
 \right\}  \right| &> | ( \Phi_i + \theta - \theta' ) \cap (-\epsilon, \epsilon) | \cdot c_2'' / 2 \\
 &> \frac{1}{2} | \Phi_i | \cdot \frac{1}{2} c_2'' =: c_2'. 
\end{aligned}
\end{equation*}
\end{proof}

\section*{Acknowledgements}
The author is grateful to Boris Solomyak for many helpful discussions and comments.

\end{document}